\documentclass[12pt]{amsart}
\usepackage[hmarginratio={1:1}]{geometry}                
\usepackage{graphicx}
\usepackage{amssymb}
\usepackage{tikz-cd} 

\newtheorem{theorem}{Theorem}
\newtheorem{lemma}[theorem]{Lemma}
\newtheorem{proposition}[theorem]{Proposition}
\newtheorem{corollary}[theorem]{Corollary}
\newtheorem{example}[theorem]{Example}
\theoremstyle{definition}
\newtheorem{definition}{Definition}[section]

\def\de{\delta}
\def\De{\Delta}
\def\eps{\varepsilon}

\def\ga{\gamma}
\def\Ga{\Gamma}

\def\om{\omega}
\def\Om{\Omega}

\def\Si{\Sigma}

\def\Sm{\setminus}

\newcommand{\bd}[1]{{\partial{#1}}}
\newcommand{\C}{{\mathbb C}}
\newcommand{\Chat}{{\widehat{\mathbb C}}}
\newcommand{\Cpct}{{c}}

\renewcommand\d{\mathrm{d}}
\newcommand{\Dbar}{{\overline{\mathbb D}}}
\newcommand{\D}{{\mathbb D}}

\newcommand*\Laplace{\mathop{}\!\mathbin\bigtriangleup}
\newcommand*\napla{{\mathop{}\!\mathbin\bigtriangledown}}
\renewcommand\Sm{{\setminus}}
\DeclareMathOperator\Arg{Arg}
\DeclareMathOperator\e{e}
\DeclareMathOperator\Pc{{Pc}}
\DeclareMathOperator\sing{{sing}}
\newcommand{\ov}[1]{{\overline{#1}}}

\newcommand\whK{{\widehat{K}}}
\newcommand\whD{{\widehat{D}}}
\newcommand\whDe{{\widehat{\Delta}}}
\newcommand\whg{{\widehat{g}}}
\newcommand\oom{{\overline{\omega}}}
\newcommand\oo{{o}}

\newcommand\Ray{{\mathcal{R}}}
\newcommand{\Sen}{{{\mathbb S}}}
\newcommand{\R}{{\mathbb R}}
\newcommand{\T}{{\mathbb T}}
\newcommand\whh{{\widehat{h}}}
\newcommand\whom{{\widehat{\omega}}}
\newcommand\whOm{{\widehat{\Omega}}}
\newcommand\whphi{{\widehat{\pi}}}

\newcommand\whtau{{\widehat{\tau}}}
\newcommand\whR{{\widehat{R}}}
\newcommand\whW{{\widehat{W}}}

\newcommand\whXi{{\widehat{\Xi}}}
\newcommand\whY{{\widehat{Y}}}

\newcommand\mapfromto[3]{{\ensuremath{#1: #2 \to #3}}}

\newcommand{\ALIGN}{\begin{align*}}
\newcommand{\ENDALIGN}{\end{align*}}
\newcommand{\ENUM}{\begin{enumerate}}
\newcommand{\ENUMa}{\begin{enumerate}[a.]}
\newcommand{\ENUMA}{\begin{enumerate}[A.]}
\newcommand{\ENUMAF}{\begin{enumerate}[\bf A.]}
\newcommand{\ENUMi}{\begin{enumerate}[i)]}
\newcommand{\ENDENUM}{\end{enumerate}}
\newcommand{\ITMZ}{\begin{itemize}}
\newcommand{\ENDITMZ}{\end{itemize}}
\newcommand{\REFEQN}[1] { \begin{equation}\label{#1} }
\newcommand{\ENDEQN}{\end{equation}}
\newcommand{\THM}{\begin{theorem}}
\newcommand{\EXA}{ \begin{example}}
\newcommand{\REFEXA}[1] { \begin{example}\label{#1} }
\newcommand{\ENDEXA}{\end{example}}
\newcommand{\MTX}{ \begin{matrix}}
\newcommand{\ENDMTX}{ \end{matrix}}
\newcommand{\REM}{ \begin{remark}}
\newcommand{\ENDREM}{\end{remark}}
\newcommand{\REFTHM}[1] { \begin{theorem}\label{#1} }
\newcommand{\RREFTHM}[2] { \begin{theorem}[#1]\label{#2} }
\newcommand{\ENDTHM}{\end{theorem}}
\newcommand{\REFNTH}[1] { \begin{newthm}\label{#1} }
\newcommand{\ENDNTH}{\end{newthm}}
\newcommand{\REFPROP}[1]{\begin{proposition}\label{#1} }
\newcommand{\RREFPROP}[2]{\begin{proposition}[#1]\label{#2} }
\newcommand{\PROP}{\begin{proposition}}
\newcommand{\ENDPROP}{\end{proposition} }
\newcommand{\REFDEF}[1]{\begin{definition}\label{#1} }
\newcommand{\RREFDEF}[2]{\begin{definition}[#1]\label{#2} }
\newcommand{\RDEF}[1]{\begin{definition}[#1]}
\newcommand{\DEF}{\begin{definition}}
\newcommand{\ENDDEF}{\end{definition} }
\newcommand{\REFLEM}[1]{\begin{lemma}\label{#1} }
\newcommand{\RREFLEM}[2]{\begin{lemma}[#1]\label{#2} }
\newcommand{\LEM}{\begin{lemma}}
\newcommand{\ENDLEM}{\end{lemma} }
\newcommand{\REFCOR}[1]{\begin{corollary}\label{#1} }
\newcommand{\COR}{\begin{corollary}}
\newcommand{\ENDCOR}{\end{corollary} }
\newcommand{\CONJ}{\begin{conjecture}}
\newcommand{\REFCONJ}[1]{\begin{conjecture}\label{#1}}
\newcommand{\RREFCONJ}[2]{\begin{conjecture}{#1}\label{#2}}
\newcommand{\ENDCONJ}{\end{conjecture} }
\newcommand{\REFDEFTHM}[1] { \begin{defthm}\label{#1} }
\newcommand{\ENDDEFTHM}{\end{defthm}}

\newcommand{\defref}[1]{Definition~\ref{#1}}

\newcommand{\lemref}[1]{Lemma~\ref{#1}}
\newcommand{\thmref}[1]{Theorem~\ref{#1}}
\newcommand{\propref}[1]{Proposition~\ref{#1}}

\newcommand{\secref}[1]{Section~\ref{#1}}

\newcommand{\PROOF}{\begin{proof}}
\newcommand{\ENDPROOF}{\end{proof}}

\renewcommand\Im{\mathrm{Im}}

\title{Conformal renormalization of compact sets}
\author{Carsten Lunde Petersen and Filip Samuelsen}
\date\today
\begin{document}
\begin{abstract}
This paper develops a conformal renormalization scheme for compact sets $K \subset \C$. 
As one application of the conformal renormalization scheme we prove that for every isolated non-trivial connected component 
$E \subset K$ there exists a conformal homeomorphism $\phi$ mapping a neighbourhood of $E$ into $\C$ 
such that the equilibrium measure on $K$ restricted to $E$ equals the scaled push-forward by $\phi^{-1}$ of 
the equilibrium measure on $\phi(E)$. 

Moreover the proof shows that the condition of connectedness of $E$ can be relaxed considerably. We also introduce an inverse to the procedure of conformal renormalization, which allows one to reconstruct $K$ from its conformal renormalizations.
\end{abstract}

\maketitle

\section{Introduction}
Equilibrium measures on non-polar compact sets in $\C$ are a corner stone of potential theory in the complex plane. 
For a thorough introduction to this field, see e.g. the monograph by Randsford, \cite{ransford}. 

The equilibrium measure $\om=\om_K$ on a given non-polar compact set $K\subseteq \C$ 
is directly related to the capacity of $K$ 
and the Green's function (with pole at infinity) for the unbounded connected component 
$D$ of $\Chat\Sm K$, usually denoted by $g_D$. 
Moreover, if $K$ is connected both are directly related to the uniformisation of $D$.

Indeed the Green's function $g=p_\om-I$, where $p_\om$ is the potential for $\om$ and $I$ is its energy. 
The capacity of $K$ is $\e^I$, see also the Notation and Preliminaries section below for further details. 
Conversely the equillibrium measure $\om$ is $\frac{1}{2\pi}$ times the Laplacian derivative of $g$ in the sense of distributions.

In the special case where $K$ is a non-trivial (i.e. not at singleton) connected compact subset $K\subset\C$ 
the Riemann uniformization theorem provides a (unique) isomorphism, 
a biholomorphic map {\mapfromto {\psi = \psi_D} D {\Chat\Sm\Dbar}} with $\psi(z)/z\to a>0$ as $z\to\infty$. 
The Green's function $g=g_D = \log|\psi|$ is harmonic on $D\Sm\{\infty\}$ 
and extends to $\C$ as a subharmonic function by $g(z) = 0$ on $K$. 
The map $\psi$ directly gives the capacity $\Cpct(K)=1/a$. 
The equilibrium measure $\om=\om_K$ with support $\bd{D}\subseteq\bd{K}$ is the push-forward of the 
normalized standard Lebesgue measure on $\Sen=\bd{\D}$ under the radial limit of $\psi^{-1}$. 

In the case where $K$ is not connected, but say has at least two non-trivial connected components $K_1$ and $K_2$ each intersecting $\bd{D}$ the Green's function still equals $\log|\psi|$ near $\infty$ 
for some univalent map $\psi = \psi_D$ with $\psi(z)/z\to a = \frac{1}{\Cpct(K)}$ 
as $z\to\infty$, but $\psi$ is not a uniformization of $D$ (see also \lemref{existence_and_uniqueness_of_psi}). 
One may ask if there is a relation between the restriction of the equilibrium measure $\om=\om_K$ on $K$ and 
the equilibrium measure $\om_j$ on $K_j$, for $j=1,2$. 
A few examples show that there is not a tangible general such relation. 
One may then ask the somewhat more basic question: 

Are there compact sets $\whK_1, \whK_2\subseteq \C$ and homeomorphisms {\mapfromto {\varphi_j}{\whK_j}{K_j}}, for $j=1,2$ 
such that the restriction of $\om$ to $K_j$ equals the scaled push-forward by $\varphi$ of the equilibrium measure $\whom_j$ on $\whK_j$, i.e.
$\om|_{K_j} = \om(K_j)\varphi_*(\whom_j)$, for $j=1,2$?

In this paper we develop a notion of conformal renormalization of compact sets in the complex plane. 
As a consequence, we show that the answer to the above question is indeed yes, one may even choose the homeomorphisms 
$\varphi_j$ above to be conformal on a neighbourhood of $K_j$. 

\THM \label{thm_1}
Let $K\subset\C$ be any non-thin compact subset with $D$, $g$ and $\om$ as above.  
Let $l>0$, let $W_1, \ldots, W_n$ be the bounded connected components of $\C\Sm g^{-1}(l)$ 
and let $K_j:=K\cap W_j$ for $j= 1, \ldots, n$. 
Then there exist univalent maps {\mapfromto {\phi_j} {W_j} \C} such that $ \whK_j := \phi(K_j)$ 
are non-thin compact sets with corresponding Green's functions 
\REFEQN{normalized_greensfunction}
g_j = \frac{1}{\om(K_j)}g \circ \phi_j^{-1}. 
\ENDEQN
Moreover $\phi_j$ is unique up to post-composition by an affine map and thus $\whK_j$ is unique up to affine transformation.
\ENDTHM
\COR Let $\whom_j$ denote the equilibriums measure for $\whK_j$ then 
\REFEQN{normalized_pushforward}
\whom_j = \frac{1}{\om(K_j)}\phi_*(\om|_{K_j}). 
\ENDEQN
\ENDCOR

We shall say that a compact set $K$ is full if its complement is connected. 
Clearly for a non-thin and compact set $K\subset\C$ the unbounded connected component $D$ of $\Chat\Sm K$ 
is simply connected if and only if the Green's function for $D$ has no positive critical value (see also \lemref{existence_and_uniqueness_of_psi}). 
We shall say that $K$ is centered if $\psi_D(z) = \frac{1}{\Cpct(K)} z+ O(1/z)$.
The conformal renormalizations of $K$, which are to be defined below, do not depend on size and positon, i.e. are 
preserved under affine maps. The capacity $\Cpct(K)$ fixes $K$ in the affine class of $K$ up to rotations and translations, and adding centeredness furthermore fixes $K$ up to rotation. 

Let $K\subset\C$ be a non-thin compact subset with $D$ multiply connected. 
Here in the Introduction we shall assume the maximal critical value $L$ of the Green's function $g$ for $D$ is simple, 
i.e. there is a unique simple critical point $z_0$ with $g(z_0) = L$. 
The reader will find the more general cases unfolded in  \secref{sec:generalisations} below.

Let $\Ray_t$ denote the smooth external ray for $K$ of argument $t\in\T$ (See definition \ref{smooth_rays}), and let $t_1< t_2 < 1+t_1$ be such that $z_0 \in \mathbb{C}$ is a limit point of $\Ray_{t_1}$ and $\Ray_{t_2}$. Furtheremore let $\Om_1, \Om_2$ be the two connected components of $\C\Sm(\Ray_{t_1}\cup\{z_0\}\cup\Ray_{t_2})$ with $\Om_1$ containing the external rays in the interval $]t_1, t_2[$ 
and $\Om_2$ containing the external rays in the interval $]t_2, 1+t_1[$.
Let $K_j := K\cap \Om_j$, then $K = K_1\cup K_2$ because $K = g^{-1}(0)$ by non-thinness. 
Let $\De_1 := t_2-t_1$ and $\De_2 := 1+t_1-t_2$ be the interior opening angles at $\infty$ of $\Om_1$ and $\Om_2$ respectively. 

\THM \label{theorem_2}
For $j\in{1,2}$, there exist unique univalent maps (uniformizations) {\mapfromto {\phi_j}{\Om_j}\C} such that 
\ENUM
\item
$\whK_j:= \phi_j(K_j)$ are centered and have capacity $1$, 
\item
the Green's functions $\whg_j$ for $\whD_j := \Chat\Sm\whK_j$ 
satisfy $\whg_j(z) = g(\phi_j^{-1}(z))/\De_j$,
\item
$\whom_j = (\phi_j)_*(\om_{|K_j})/\De_j$, 
where $\whom_j$ denote the equilibrium measures for $\whK_j$, 
\item
and finally, the continuous extension of $\phi_j$ to $\ov{\Om}_j$ sends the arc $\Ray_{t_1}\cup\{z_0\}\cup\Ray_{t_2}$ 
into the smooth external ray $\Ray_0(\whK_j)$.
\ENDENUM
\ENDTHM
Note that taking $t_1'=t_2$ and $t_2'=1+t_1$ exchanges all labels $1$ and $2$.
We call $\whK_1$ and $\whK_2$ conformal renormalizations of $K$.
There are many different compact sets $K'$ with the same renormalizations as a given set $K$. 
However, given the numbers $\Cpct(K), L, t_1, \De_1$ and non-thin and centered compact sets $\whK_1, \whK_2$ of capacity $1$, 
we can reconstruct $K$ up to translation and in particular reconstruct $K$ if $K$ is centered. Indeed, we prove the following converse theorem.
\THM \label{inv_ren_thm}
Given $t_1\in\T$ and $C,L, \De_1 > 0$ with $\De_1 < 1$ and non-thin centered compact sets $K_1, K_2$ of capacity $1$ with corresponding Green's functions $g_1$ and $g_2$ satisfying 
$L > \sing(g_1) \De_1$ and $L > \sing(g_2) (1-\De_1)$. Then  
there exists a unique non-thin and centered compact set $K$ of capacity $C$ such that 
the Green's function $g$ for $D$ has a unique simple critical point $z_0$ of maximal potential $L$ with external arguments 
$t_1$ and $t_2= t_1+\De_1$ and such that the conformal renormalizations of $K$ are $K_1$ and $K_2$. 
\ENDTHM

\section{Notation and preliminaries} \label{sec:not_and_pre}
We need the following notions and results from potential theory.
\begin{definition}[Potential and energy] 
	Let $\mu$ be a Borel probability measure with compact support $S(\mu): = \mathrm{supp}\,\mu\subset\C$. 
	We follow the sign-convention of Randsford, \cite{ransford} and define the \textit{potential} function $p_\mu : \mathbb{C} \to [-\infty, \infty)$ 
	of $\mu$ by 
	\[
		p_\mu \left(z\right) := \int_\mathbb{C} \log |z-w| \mathrm{d} \mu\left(w\right)
	.\] 
	And the \textit{energy} $I$ of $\mu$, as 
	\[
		I\left( \mu \right) := \int_\mathbb{C} p_\mu\left(z\right) \mathrm{d} \mu\left(z\right)  = 
		\iint\limits_{\mathbb{C} \times \mathbb{C}} \log \left| z-w \right| 
		\mathrm{d}\mu\left(w\right) \mathrm{d}\mu\left(z\right)  \in [-\infty,\infty[
	.\]  
\end{definition}
\RDEF{Capacity}
	For $K$ a bounded subset of $\C$ the \textit{(logarithmic) capacity} of $K$ is defined by 
	\begin{align}
	\Cpct\left(K\right) := \sup_\mu e^{I\left(\mu\right)} \in [0,\infty[ \label{eq:capacity_definition} 	
	.\end{align} 
	where the supremum is taken over all Borel probability measures $\mu$ with compact support $S(\mu) \subseteq K$.
\ENDDEF
The capacity is a set function which satisfies the following elementary properties
\begin{itemize}
	\item $E \subseteq K \implies c\left(E\right) \leq c\left(K\right)$.
	\item $c(\alpha K + \beta) = |\alpha| \cdot c(K)$ for every $\alpha, \beta \in \mathbb{C}$.
\end{itemize}
A subset $E$ of $\mathbb{C}$ is called polar if $I\left(\mu\right)=-\infty$ for every Borel probability measure $\mu$ with 
compact support $S(\mu)\subseteq E$, or equivalently if $c(E)=0$.
We say that a property holds nearly everywhere (n.e.) on a subset $E$ of $\mathbb{C}$ if it holds everywhere on $E \setminus K$ for some Borel polar set $K$.

The polar sets play the role of negligible sets in potential theory, much like sets of measure zero in measure theory, and every Borel polar set has Lebesgue measure zero.

The complement mapping in $\Chat$ defines a natural \emph{duality} 
between the set of compact subsets of $\C$ and the set of domains in $\Chat$ containing $\infty$. 
That is for $K\subset\C$ a compact subset we define $K^*$ as the connected component $D$ 
of $\Chat\Sm K$ containing $\infty$. And for $D\subset\Chat$ with $\infty\in D$ we denote by 
$D^*$ the compact subset $K = \Chat\Sm D \subset\C$. 
The double dual $K^{**}$ of a compact set $K$ equals the polynomial convex hull $\Pc(K)$ of $K$, 
i.e. $K$ with holes filled.  
Recall from the introduction that when $K= K^{**} =\Pc(K)$ we say that $K$ is full.
The double dual $D^{**}$ of a domain $D\subset\Chat$ containing $\infty$ equals $D$. 
It follows that the duality defines a bijection between the set of full compact subsets of $\C$ and 
the set of subdomains of $\Chat$ containing $\infty$. 

For $K\subseteq \C$ a compact set with $\Cpct(K)> 0$ the supremum in \eqref{eq:capacity_definition} is attained 
at a unique Borel probability measure $\om_K$, called the equilibrium measure of $K$. 
The outer boundary of $K$, i.e. the common boundary of $D=K^*$ and $K^{**} = \Pc(K)$ 
equals the support of $\om_K$. 
Thus the capacity of $K$ equals the capacity of the outer boundary of $K$, 
$\Cpct\left(K\right) = \Cpct\left(\partial\Pc(K)\right)$. 

\begin{definition}[Green's function with pole at $\infty$] \label{definition_greens_function}
Let $D\subsetneq \Chat$ be a proper subdomain containing $\infty$ and with non-polar complement $K = D^*$. 
The \emph{Green's function} for $D$ (with pole at $\infty$) is the unique mapping $g=g_D: \Chat \to [0,\infty]$, such that 
\begin{enumerate}
	\item $g$ is subharmonic on $\C$,
	\item $g$ is harmonic on  $D \setminus \{ \infty \}$, 
	\item $g_D\left( \infty \right) = \infty$,  \[
	g_D\left(z\right) = \log |z| + O\left(1\right) \; \mathrm{as} \; z \to \infty .\] 
	\item $g$ is bounded outside each neighbourhood of $\infty$ and equals $0$ on $\Chat\Sm\ov{D}$.
	\item $g_D\left(z\right) \to 0$ as $z \to \zeta$, for nearly every  $\zeta \in \partial D$.
\end{enumerate}
\end{definition}

In view of the duality we shall often abuse notation and also call $g$ the Green's function of $K$.

The existence of a Green's function for any domain $D\subset\Chat$, 
which contains $\infty$ and which has non-polar complement $K = D^*$, 
is an immediate consequence of Frostmans theorem \cite[Thm.~3.3.4]{ransford}. 
Indeed the function 
	\[
	g(z) = \begin{cases} p_{\om_K}\left(z\right) - I\left(\om_K\right)   &  \mathrm{for}  \; z \in \C \\
		\infty & \mathrm{for} \;  z=\infty
		\end{cases}
	\] 
satisfies the requirements, where $\om_K$ denotes the equilibrium measure for $K$. The uniqueness of such functions is an elementary consequence of the extended maximum principle for harmonic functions. 

\begin{theorem}\cite[4.2.4]{ransford}
Let $D$ with $\infty\in D\subsetneq\Chat$ be a proper subdomain with dual full compact set $K=\Chat\Sm D$ and let $\zeta\in\partial D = \partial K$. 
Then the following are equivalent:
\ITMZ
\item
$K$ is non-polar and $ \lim_{z \to \zeta} g_D(z) = 0$,
\item
the point $\zeta$ is a Dirichlet regular boundary point of $D$,
\item
$K$ is non-thin at $\zeta$.
\ENDITMZ
\end{theorem} 

If every $\zeta \in \partial D$ is a Dirichlet regular boundary point, then $D$ is called a Dirichlet regular domain. 
Correspondingly, if every boundary point of $K$ is non-thin, then $K$ is called non-thin. 
In view of the above cited theorem, $D$ is Dirichlet regular if and only if $K$ is non-thin and in particular non-polar.

In order to keep things simple we shall in the following consider only full non-thin compact sets $K$ and 
thus Dirichlet regular dual domains $D = K^*$. 
For $g$, the Green's function for $D$, we denote by critical point for $g$ any point $c\in D$ 
such that $\napla g(c) := \left(\frac{\partial g}{\partial x}+i\frac{\partial g}{\partial y}\right)(c) = 0$, 
and critical value the image $v=g(c)>0$ of any critical point. 
We denote by a singular value of $g$ any critical value of $g$ as well as the value $0= g^{-1}(K)$ 
and we denote by $\sing(g)$ the set of singular values of $g$. 
It follows from the definition of the potential $p_{\om_K}$ that the critical points of $g$ 
are contained in the convex hull of $K$ and so $\sing(g)$ is bounded.

Our notion of conformal renormalization of a full non-thin disconnected compact set $K$ relies on the fact that the Green's function $g$ for the 
complement $D = \Chat\Sm K$ defines a translation structure with purely imaginary translations on 
$\C\Sm g^{-1}(\sing(g))$, as well as the fact that the corresponding monodromy at $\infty$ is given by the translation $z\mapsto z+i2\pi$. 

As a consequence, the Green's function $g$ of $D$ is the real part of the logarithms of a univalent map $\psi$
mapping the topological disk $U\subseteq D$ around $\infty$ bounded by the maximal singular values level set for $g$ 
to a round disk around $\infty$. For completeness we include a proof.

\begin{lemma} \label{existence_and_uniqueness_of_psi}
	Let $D$ with $\infty\in D\subsetneq \Chat$ be a Dirichlet regular domain, and let $g=g_D:\Chat\to [0,\infty]$ be the Green's function for $D$. 
	Let $L=\max \{v | v\in\sing(g)\}$, $U= \{z | g(z) > L \}$ and $V=\{z | \log|z| > L \}$. Then there exists a 
	biholomorphic mapping $\psi:U \to V$ such that 
	\begin{align}
	|\psi(z)|=e^{g(z)} \label{logpsi_equals_g_of_z}
	\end{align}
Furthermore $\vline \frac{\psi(z)}{z} \vline \to \frac{1}{c\left(\partial D \right)}$. In particular, $\psi$ is unique up to rotation.
\end{lemma}
\begin{proof}
	Since $U$ is a simply connected subset of $\Chat$, the existence of a conformal mapping $\psi: U \to V$ follows from Riemann uniformization theorem. We can normalize the conformal isomorphism $\psi$ by requiring that 
	\begin{align*} 
	\psi(z) = az + O(1),\qquad z\to \infty
\end{align*}
for some $a>0$. Since $\psi(z)$ is holomorphic then $\log|\psi(z)| - L$ is harmonic on $U \setminus \{ \infty \}$ and bounded outside each neighbourhood of $\infty$. Also $\log|\psi(z)| \to L$ for every $\zeta \in \partial D$, so by uniqueness of the Green's function $g_U = \log|\psi(z)| - L$.

Notice also that $g_D(z)-L$ is harmonic on $U \setminus \{ \infty \}$ and bounded outside each neighbourhood of infinity, so by continuity of $g_D$, we find that $\left( g_D(z)-L \right) \to 0$ as $z\to \zeta$ for every $\zeta \in \partial U$ and $g_D(z)-L = \log|z| + O(1)$ as $z \to \infty$. So by uniqueness of the Green's function $\log|\psi(z)| - L=g_U(z)=g_D(z)-L$. Therefore 
\begin{align}
	\log|\psi(z)|=g_D(z)
\end{align}
so $|\psi(z)| = e^{g(z)}$ and since $g(z) = \log|z| - I(\omega_{\partial D}) + o\left(\frac{1}{|z|}\right)$ as $z \to \infty$ then $$|\psi(z)| = \frac{|z|}{c\left(\partial D \right)} + O\left(1\right) \quad \quad \text{as } z \to \infty$$
\end{proof}
We choose $\psi$ so that $\psi(z)/z\to 1/\Cpct(K)$ as $z\to\infty$ and define the normalised argument function $\tau: U \setminus \{ \infty \} \to \T = \R / Z$ by observing that 
$\psi$ is of the form $\psi\left(z\right)=e^{g\left(z\right) + i 2 \pi \tau \left(z\right)}$, i.e. $2\pi \tau$ is a harmonic conjugate to $g$.

The equipotential curve of potential $l>0$ is by definition the level set 
$$g^{-1}(l) =\{ z | g\left(z\right) = l\}.
$$
 A field-line for $g$ is a smooth curve {\mapfromto \ga {\;]a,b[} D}, $0\leq a < b \leq\infty$, 
 which is a solution of the differential equation 
$$
 \ga'(t) = \frac{\napla g(\ga(t)) }{||\napla g(\ga(t))||^2} .
$$ 
Here as usual the overline indicates complex conjugation. This choice of parameterization, relates nicely to the potential since $\frac{\d}{\d t} g(\gamma(t)) = 1$ so  $g(\gamma(t))=t + C$ where $C$ is a constant of integration. 
We shall always take $C=0$, so that the field-line is parametrized by the potential. 
\begin{definition}[Smooth external rays] \label{smooth_rays}
The \textit{smooth external ray} $\Ray_t=\Ray_{t,K}$ for $K$ of argument $t\in\T$ is the range of the maximal field-line for $g$, which contains the curve $\psi^{-1}(\{\e^{l+i2\pi t}| l > L\})$. The domain of $\Ray_t$ is $]a(t),\infty[$, where $a(t)\geq 0$.
And the points $z \in \Ray_t$ are said to have \emph{external argument} $t$.  
\end{definition}

Let $\gamma: {\;]a,b[} \to D$ be any maximal field-line. 
Then $\ga(b):= \lim_{t \to b} \gamma(t)\in\Chat$ exists. We denote it the ascending landing point for $\gamma$. 
If the range of $\ga$ is a smooth external ray then $\ga(b) = \infty = b$. Otherwise $\ga(b)$ is a critical point for $g$. 
If $0<a$ then $\ga(a):= \lim_{t \to a} \gamma(t)\in\C$ exists and is a critical point for $g$. 
We denote it the descending landing point of $\ga$. 
If $\ga$ is a ray $\Ray_t$ and $a=a(t) >0$ 
we also say that $\Ray_t$ bumps into the critical point $\ga(a)$.
If $a=0$, then $\lim_{t \to a} \gamma(t)$ may or may not exist, 
however the accumulation set is a subset of $\partial D$ (see also below). 
We shall not need this case in our further discussions. 
A critical point $z_0$ for $g$ of order $d-1\geq 1$ is easily seen to always be the landing point of precisely $d$ ascending field-lines and precisely $d$ descending field-lines. 
In view of the above we shall also write $\gamma: {\;[a,b]} \to D$, whenever $a>0$.

\begin{definition}[Extended external rays.]
Suppose $\Ray_t$ bumps into a critical point $z_0$ with $g(z_0) = a(t)$. An extended ray $\Ray_t^e$ of argument $t$, 
is a curve {\mapfromto {\Ray_t^e} {[l_0,\infty[} D}, $0 < l_0 \leq a(t)$ such that $g(\Ray_t^e(l)) = l$ for any $l \geq l_0$ and 
$\Ray_t^e(l) = \Ray_t(l)$ for $l>a(t)$.
\end{definition}
Besides the final ray $\Ray_t$ an extended ray $\Ray_t^e$ consists of finitely many maximal field-lines and possibly 
one segment of a field-line connecting at critical points in $\Ray_t^e$. Examples are the left and right extensions 
of $\Ray_t$ which always turns left respective right at the critical points encountered when moving from $\infty$ 
down the extended ray. However extended rays may also mix left and right turns as well as take intermediate options at non-simple critical points.

\section{Simple conformal renormalization} \label{sec:conf_ren}
Let $K$ denote a full compact, non-thin subset of the complex plane with dual domain $D$. 
Let $g:\Chat\to [0,\infty]$ denote the Green's function for $D$. 

For now let us assume that $g$ has exactly one simple critical point of maximal potential, 
i.e. there exists a point $z_0 \in D$, such that $z_0$ is a simple critical point 
and if $z$ is a critical point not equal to $z_0$ then $g(z) < g(z_0)$.

Let $L:=g(z_0)$ and let $U$, $V$, $\psi$ and $\tau$ be as 
in \lemref{existence_and_uniqueness_of_psi} and the trailing remarks of its proof.
With these assumptions the level set $g^{-1}(L)$ is a figure eight curve and two external rays, 
$\mathcal{R}_{t_1}$ and $\mathcal{R}_{t_2}$, $t_1 < t_2< 1 + t_1$ will bump at $z_0$. 
By the Jordan curve theorem the complement $\Chat\Sm(\Ray_{t_1}\cup\Ray_{t_2}\cup\{\infty, z_0\})$ 
consists of two  Jordan domains, $\Omega_1$ and $\Omega_2$, 
with $\Om_1$ containing the rays in the interval $]t_1, t_2[$ and $\Om_2$ 
containing the rays in the interval $]t_2, 1+t_1[$. 
Let $\De_1 = (t_2-t_1)$ and $\De_2=(1+t_1-t_2)$ denote the interior opening angles at $\infty$ of $\Omega_1$ and $\Om_2$ 
(see figure \ref{fig:riemann_sphere}).
We define $K_j = K \cap \Omega_j$ then $K_1 \cap K_2 = \emptyset$ and moreover $K=K_1 \cup K_2$, 
since $K$ is non-thin.

\begin{figure}[h]
	\centering
	\includegraphics{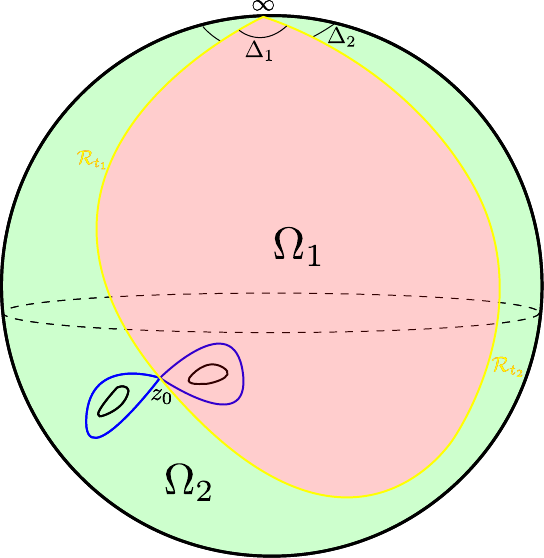}
	\caption{An illustration of a Green's function with a single simple critical point of maximal potential as viewed on the Riemann sphere. In black the boundary of the domain $D$, in blue the equipotential curve $L_{g\left(z_0\right)}$, in yellow the external rays $\mathcal{R}_{t_1} , \mathcal{R}_{t_2}$ landing at a critical point $z_0$, in light red $\Omega_1$ and in light green $\Omega_2$. $\Delta_1$ and $\Delta_2$ are the normalised interior angles between $\mathcal{R}_{t_1}$ and $\mathcal{R}_{t_2}$ at $\infty$ measured in $\Omega_1$ and $\Omega_2$ respectively.} \label{fig:riemann_sphere}
\end{figure}

We shall glue $\mathcal{R}_{t_1}$ to $\mathcal{R}_{t_2}$ equipotentially in order to form new surfaces. 
To this end let $\sim_j$ be the smallest equivalence relation on $\overline{\Omega}_j$ for which 
$x\sim_j y $ for any pair $x\in \mathcal{R}_{t_1}$ and $y \in \mathcal{R}_{t_2}$ with $g(x)=g(y)$. 
Note that we may equivalently require that $\psi(x)e^{\Delta_j 2 \pi i} = \psi(y)$.

\begin{figure}[h]
 \centering 
 \includegraphics{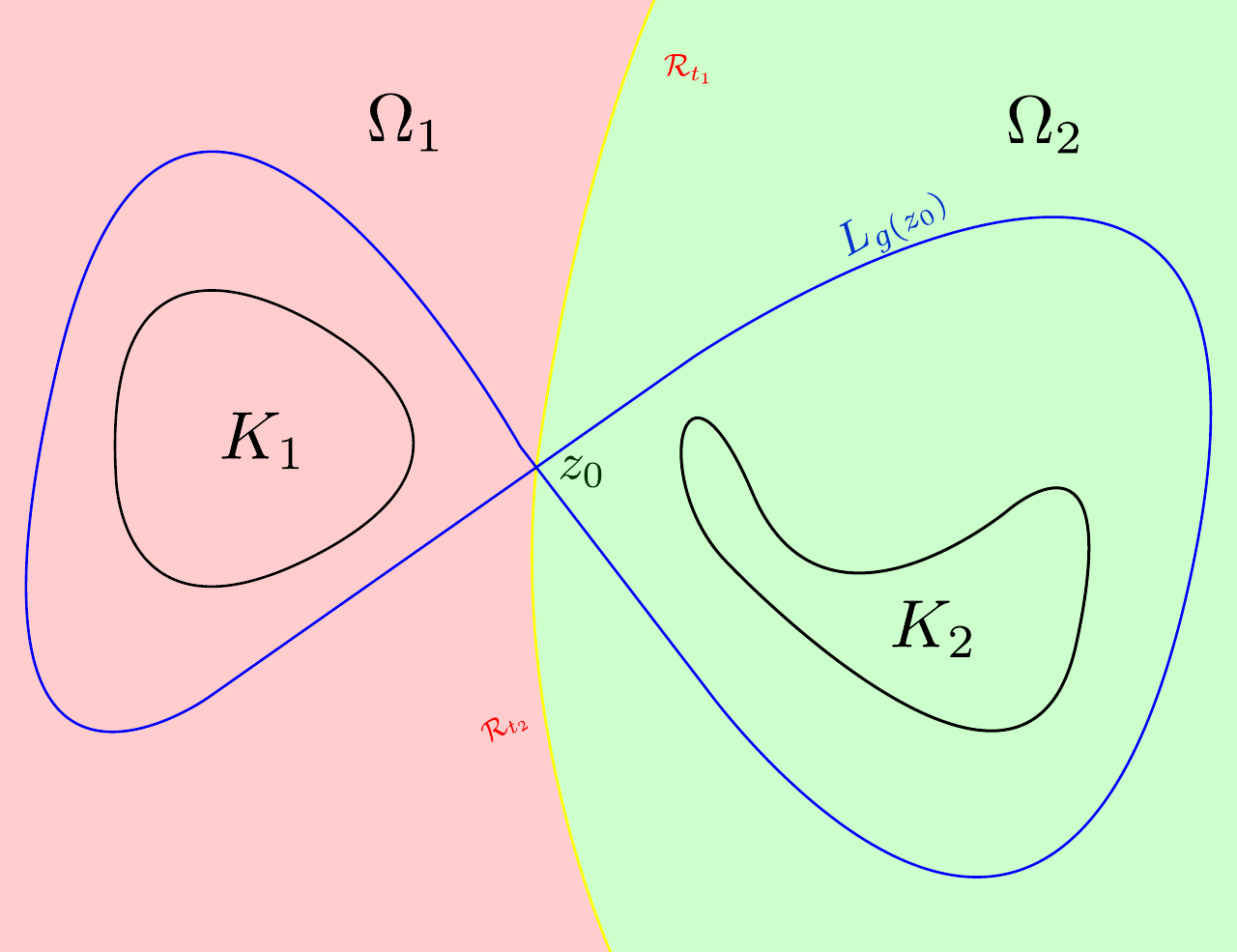}
 \caption{An illustration of a Green's function with a single simple critical point of maximal potential. In black the boundary of the domain $D$, in blue the equipotential curves $L_{g(z_0)}$, in yellow the orthogonal rays, $\mathcal{R}_{t_1},\mathcal{R}_{t_2}$ landing at $z_0$, in light red $\Omega_1$ and in light green $\Omega_2$.} \label{fig:single_simple_critical_point}
\end{figure}

Let $X_j:=\overline{\Omega}_j/{\sim_j}$ be the induced quotient space and let $\pi_j: \overline{\Omega}_j\to X_j$ be the canonical projection. 
We equip $\overline{\Omega}_j$ with the relative topology and $X_j$ with the quotient topology, so that $\pi_j$ is a quotient map. 
Evidently $X_j$ is compact, simply connected and Hausdorff.

The projection $\pi_j$ induces a canonical complex structure on $X_j$ making it conformally equivalent to the Riemann sphere. 
Indeed define $R_j:= \pi_j\left(\mathcal{R}_{t_1}\right) = \pi_j \left( \mathcal{R}_{t_2} \right) $.
\begin{lemma} \label{conform_equivalent}
	There exists a unique complex structure on $X_j$ such that the restricted projection $\pi_{j\;}\vline_{\,\Omega_j}: \Omega_j \to X_j \setminus R_j$ is a biholomorphic mapping.
\end{lemma}
\begin{proof}
	We will need 4 charts to cover $X_j$.
\begin{enumerate}
	\item We can use the inverse of the projection restricted to $\Omega_j$ as the first chart.
	\item For the second chart we first define strips 
		$$V_k = \{z \in \mathbb{C} | g(z) > g(z_0), t_k-\varepsilon < \tau(z) < t_k + \varepsilon \} \quad \text{for} \; k=1,2$$ with $\varepsilon$ sufficiently small such that $V_1 \cap V_2 = \emptyset$. Then the mappings \\ $f_j: \pi_j\left(\overline{\Omega}_j \cap (V_1 \cup V_2) \right) \to \psi(V_2)$ defined by \[
	f_j(z) = \begin{cases} \psi\left(\pi_j^{-1}(z)\right) e^{\Delta_j 2 \pi i} \quad & \text{ for } z \in \pi_j\left(V_j \cap \overline{\Omega}_j\right) \\
	\psi\left(\pi_j^{-1} (z)\right) \quad & \text{ for } z \in \pi_j\left(V_{3-j} \cap \Omega_j\right)
	\end{cases}
	\] 

are evidently homeomorphisms that cover $R_j$. It now remains to cover ${\pi_j(z_0), \pi_j(\infty)}$.
	\item We cover $\pi_j(\infty)$ by the mapping $\left(\psi(\pi_j^{-1}(z))\right)^{-\frac{1}{\Delta_j}}$. Since $ \left( \psi(z) \right)^{\frac{1}{\Delta_j}}$ identifies exactly those points which belong to the same equivalence class, then the mapping is a well defined homeomorphism.
	\item Finally to cover $\pi_j(z_0)$, we use that $z_0$ is a simple critical point of $g$. As a consequence there exists a sufficiently small neighbourhood $O_{z_0} \subseteq \mathbb{C}$ of $z_0$, $\varepsilon_1>0$ and a harmonic conjugate $h$ of $g$ on $O_{z_0}$ with $h(z_0)=0$ such that $f:=g+ih:O_{z_0}\to\D(g(z_0),\varepsilon_1)$ is a branched $2:1$ covering with critical value $g(z_0)$. 
	It follows that $f(\pi_j^{-1}): \pi_j(\overline{\Omega}_j \cap O_{z_0}) \to \mathbb{D}(g(z_0),\varepsilon_1)$ is a chart compatible with the other charts.

\end{enumerate}
Clearly these charts define a complex structure on $X_j$. Fix $j\in\{1,2\}$ and 
let $\mathcal{A}_1$ denote the above complex structure on $X_j$.
Suppose that $\mathcal{A}_2$ is another complex structure on $X_j$, for which $\pi_j\;\vline_{\;\Omega_j}$ is biholomorphic, then any chart on $X_j \setminus R_j$ would clearly have to be compatible.

Suppose that $\phi_1$ and $\phi_2$ are charts in $\mathcal{A}_1$ and $\mathcal{A}_2$ respectively, 
both covering $w \in R_j$. Let $O_w$ be an open neighbourhood of $w$ covered by both charts. 
Then their transition mapping $k:\phi_1(O_w) \to \phi_2(O_w)$ defined by $k = \phi_2 \circ \phi_1^{-1} $ 
is a homeomorphism and it is holomorphic outside the straight line segment $\phi_2(R_j \cap O_w)$, 
which is holomorphically removeable.
Therefore $k$ extends holomorphically to the entire neighbourhood and so the change of charts is holomorphic.
\end{proof}

Since $X_j$ is a compact and simply connected Riemann surface, 
it is isomorphic to the Riemann sphere by the Riemann uniformization theorem.

For $j \in \{1,2\}$, let $\Phi_j:X_j \to \Chat$ be the unique isomorphism normalized by
\begin{align}
\Phi_j(\pi_j(z)) = \left(\e^{-i2\pi t_j}\psi(z)\right)^{\frac{1}{\Delta_j}} + O\left((\psi(z))^{\frac{-1}{\Delta_j}}\right)
\qquad\textrm{as}\qquad z \to \infty, \label{normalise_conf_rocs}
\end{align}
so that in particular $\Phi_j(\pi_j(\infty))=\infty$. 

We define $\phi_j := \Phi_j \circ \pi_j$ and compact subsets $\whK_j = \phi_j\left(K_j\right)\subseteq \C$.

The normalization choosen in equation (\ref{normalise_conf_rocs}) is natural and allows for a cannonical uniformization of $K_1$ and $K_2$.

Each compact set $\whK_j$ is full and non-thin and therefore carries a unique equilibrium measure, $\widehat{\omega}_{j}$, and its dual domain, $\widehat{D}_j$, has a unique Green's function $\widehat{g_j}:\widehat{D}_j\to [0,\infty]$.

We are now ready to establish the connection between $\whK_j$ and $K$.

\begin{theorem} \label{greensfunctions_related} In the notation above \hfill
\begin{enumerate} 
\item $\widehat{g_j}(z)=\frac{1}{\Delta_j}g\left(\phi_j^{-1}(z)\right)$.
\item $\phi_j$ is unique with the following normalizations
\begin{enumerate}
\item $c\left(\whK_j\right) = 1$
\item $\whK_j$ is centered
\item $\phi_j(\Ray_{t_1}) = \phi_j(\Ray_{t_2})$ is an initial subarc of the smooth external ray for $\whK_j$ of argument $0$.
\end{enumerate}
\end{enumerate}
\end{theorem}
\begin{proof}
 (1) Note that $\frac{1}{\Delta_j}g\left(\phi_j^{-1}(z)\right)$ is a well defined function since $g$ respects the equivalence relation. Let $w \in \overline{\Omega}_j$, then as $w \to \infty$ we have
\begin{align}
\log|\phi_j(w)| &= \log\,\vline\psi(w)^{1/\Delta_j} + O\left(\frac{1}{\psi(w)^{1/\Delta_j}}\right)\vline\\
&= \log|\psi(w)^{1/\Delta_j}| + O\left(\left|\frac{1}{\psi(w)^{2/\Delta_j}}\right|\right)\\
&= \frac{1}{\Delta_j}\log|\psi(w)| + O\left(\left|\frac{1}{\psi(w)^{2/\Delta_j}}\right|\right)\\
&= \frac{1}{\Delta_j}g(w) + O\left(\left|\frac{1}{\psi(w)^{2/\Delta_j}}\right|\right) \label{green_error_term}
\end{align}
Since $|\psi(w)|=O(|w|)$ as $w\to \infty$, then $O\left(\left|\frac{1}{\psi(w)^{2/\Delta_j}}\right|\right)$ is $o(1)$ as $w\to \infty$ so  
\begin{align}
\frac{1}{\Delta_j}g(w) = \log|\phi_j(w)| + o\left(1\right) & \text{ as } w\to \infty
\end{align}
Now let $z=\phi_j(w)$ then as $w \to \infty$ we have that $z\to \infty$ so 
\begin{align}
\frac{1}{\Delta_j}g(\phi_j^{-1}(z)) = \log|z|+o(1) & \text{ as } z \to \infty \label{eq:green_for_renormalisering}
\end{align}
And since $\frac{1}{\Delta_j}g(\phi_j^{-1}(z))$ is harmonic and vanishes exactly at $\partial  \widehat{D}_j$ it must, by uniqueness of the Green's function, be the Green's function for the component containing $\infty$ of $\Chat \setminus \whK_j$.

(2) Suppose we have two uniformizations $\phi_j$ and $\phi_j'$ univalent on $\Omega_j$, extending continuously to 
$\Ray_{t_1}\cup\Ray_{t_2}$ with $\phi(\Ray_{t_1}(l)) = \phi(\Ray_{t_2}(l))$. Then $k = \phi_j' \circ \phi_j^{-1}$ is a homeomorphism from $\Chat$ to itself and it is holomorphic outside the arc $\phi_j(\Ray_{t_1})$, which is holomorphically removeable. Therefore $k$ extends holomorphically to all of $\Chat$ and so it must be a mobius transformation. By (1) then $k(\infty)=\infty$ so $k$ must be affine. Now (a), (b) and (c) fixes $\phi$ in the affine class.

(a) It follows immediately from equation \eqref{eq:green_for_renormalisering} that the energy for the equilibrium measure of $\whK_j$ is $0$. Hence the capacity is 1.

(b) It follows from equation \eqref{green_error_term} that $\log|\phi_j(w)| = \frac{1}{\Delta_j}g(w) + O\left(\left|\frac{1}{\psi(w)^2}\right|\right)$ and therefore
$$
\widehat{g}_j(z)= \log|z| + O\left(\left| 1/z^2 \right|\right)
$$
so we have that $\left|e^{\widehat{g}_j(z)}\right| = |z| + O(\left|1/z\right|)$ as $z \to \infty$, i.e. $\whK_j$ is centered.
(c) Property (a) and (b) fixes $\phi_j$ up to rotation, it then follows from the normalization in equation (\ref{normalise_conf_rocs}) that $\phi_j(\Ray_{t_1})=\phi_j(\Ray_{t_2})=\Phi_j(R_j)$ is an initial subarc of the smooth external ray for $\whK_j$ of argument $0$.
\end{proof}
We will call the triple $(\Omega_j, \phi_j, \whK_j)$ for a \textit{simple normalized conformal renormalizations} of $K$. If $f$ is an affine map then the triple $(\Omega_j, f \circ \phi_j, f(\whK_j))$ is called a \textit{simple conformal renormalizations}.
We remark that in this simple case, the image of the critical point, $z_0$, under each $\phi_j$ is no longer a critical point. 
In fact the maximal critical values are strictly smaller than $g_j(\phi(z_0)) = L/\Delta_j$ for each $j=1, 2$. 

\section{Generalisations} \label{sec:generalisations}
In this section we show how to generalize \thmref{greensfunctions_related} aswell as the construction leading up to it. Note that while external rays are always disjoint, two or more extended rays may intersect. This motivates the following definition.
\DEF
Let $\Ray_{t_1}$ and $\Ray_{t_2}$ be two external rays bumping into the same critical point with $t_1 < t_2 < 1 + t_1$. 
And let $\Ray_{t_1}^e, \Ray_{t_2}^e$ be extended rays with domain $[l_0,\infty[$, where $l_0>0$. 
We call $(\Ray_{t_1}^e, \Ray_{t_2}^e)$ an \emph{extended ray-pair} iff $\Ray_{t_1}^e(l_0) = \Ray_{t_2}^e(l_0)$ 
\ENDDEF
Possibly increasing $l_0$ we may assume that $\Ray_{t_1}^e(l) \not= \Ray_{t_2}^e(l)$ for $l>l_0$. 
Then the common end-point $\Ray_{t_1}^e(l_0) = \Ray_{t_2}^e(l_0)$ is a critical point for $g$.

Let $t_1<t_2<1+t_1$ be arguments of external rays co-landing on a simple critical point $z_0$ of maximal potential $L$, 
as in \secref{sec:conf_ren} and let $\Ray_{t_1}^e, \Ray_{t_2}^e$ be the unique extended rays with domain $[L,\infty[$. 
Then both $(\Ray_{t_1}^e, \Ray_{t_2}^e)$ and $(\Ray_{t_2}^e, \Ray_{t_1}^e)$ are examples of extended ray-pairs. 

\begin{definition} \label{def:conf_renormalization}
A \emph{conformal renormalization} of $K$ is a triple $\left(\Omega, \phi, \whK \right)$ such that 
\ENUM
\item 
$\Om$ is a domain in $\C$ with boundary a finite collection of extended ray-pairs 
$(\Ray_{t_1}^e, \Ray_{t_2}^e)$, \ldots, $(\Ray_{t_{2n-1}}^e, \Ray_{t_{2n}}^e)$ where 
$t_1<t_2<t_3<\ldots < t_{2n-1} < t_{2n} < 1+t_1$.
\item $\phi: \Omega \to \C$ is univalent. 
\item $\whK = \phi(K \cap \Omega)$.
\item $\widehat{g}(z)=\frac{1}{\omega(K \cap \Omega)}g\left(\phi^{-1}(z)\right)$, 
where $\omega$ is the equilibrium measure for $K$, 
$g$ is the Green's function for $D= \Chat \setminus K$ 
and $\widehat{g}$ is the Green's function for $\widehat{D}= \Chat \setminus \whK$.
\ENDENUM
Moreover we call $(\Omega, \phi, \whK)$ a \emph{normalized conformal renormalization} of $K$ 
when $\whK$ is centered and $c(\whK)=1$.
\end{definition}
Note that in the above definition
$$
\omega(K \cap \Omega) = 1 - \sum_{j=1}^n (t_{2j}-t_{2j-1}).
$$
We shall refer to $\whK$ as a conformal renormalization of $K$, when $\Omega$ and $\phi$ is understood from the context. It follows from theorem \ref{greensfunctions_related} and the fact that $\Delta_j = \omega(K_j)$ that $\whK_1$ and $\whK_2$ are normalized conformal renormalizations of $K$. 

It is easy to construct examples of compact sets $K$ with two distinct renormalizations $(\Omega_1, \phi_1, \whK_1)$, 
$(\Omega_2, \phi_2, \whK_2)$, with $\Om_1\cap\Om_2=\emptyset$ and $\whK_1=\whK_2$. 
But given $K$ and $\Om$ the map $\phi$ and the set $\whK$ are essentially unique.

\REFTHM{phi_K_uniquenes}
Let $(\Omega, \phi, \whK)$ be a conformal renormalization of $K$. 
Then $\whK$ is unique up to affine transformation and equivalently 
the univalent map $\phi$ is unique up to post-composition by an (the same) affine map. 
Moreover $\phi$ extends continuously to the boundary of $\Om$ with $\phi(\Ray_{t_{2j-1}}^e) = \phi(\Ray_{t_{2j}}^e)$ 
either a segment of a smooth external ray or of an extended external ray for each $j$.
\ENDTHM

\PROOF
Suppose $(\Omega , \phi_1, \whK_1)$ and $(\Omega , \phi_2, \whK_2)$ both are conformal renormalisations of $K$ with the same domain $\Om$. Define  {\mapfromto {f = \phi_2 \circ \phi_1^{-1}}{\phi_1(\Om)}{\phi_2(\Om)}}. 
We must prove that $f$ is affine. 
We first note that $f$ maps field-lines to field-lines and in particular external rays to external rays, 
since $\whg_1(z) = \whg_2(f(z))$ and for the same price $2\pi\whtau_1(z)-2\pi\whtau_2(f(z)$ is locally constant, 
being harmonic conjugates of the Green's functions. 
Also $\whtau_j(z)-\tau(\phi_j^{-1}(z))$, $j=1,2$ are locally constant by property (4). 
It follows that the continuous extension of $\phi_j$ to any of the extended ray-pairs bounding $\Om$, 
maps the pair into the same ray or possibly onto an extended ray. 
But then $f$ extends to a homeomorphism from $\C$ to $\C$.
Finally since any extended ray is holomorphically removeable then $f$ is an isomorphism from $\C$ to $\C$ and therefore $f$ is affine.
\ENDPROOF

\begin{theorem} \label{push_forward_measure}
Let $(\Omega, \phi, \whK)$ be a conformal renormalization of $K$ then
$$\widehat{\omega} = \frac{1}{\omega(K \cap \Omega)} \phi_*(\omega \, |_{(K \cap \Omega)})$$ where $\widehat{\omega}$ is the equilibrium measure for $\whK$ and $\omega$ is the equilibrium measure for $K$.
\end{theorem}
\begin{proof}
Let $C^\infty_c(U)$ denote be space of compactly supported $C^\infty$-functions on $U \subseteq \mathbb{C}$.

Since $(\Omega, \phi, \whK)$ is a conformal renormalizations then $p_{\widehat{\omega}}(z) = \frac{1}{\omega(K \cap \Omega)} \left( p_{\omega}\left( \phi^{-1} \right) + I(\omega) \right)$, so therefore $$\Laplace p_{\widehat{\omega}} = \Laplace \left(\frac{1}{\omega(K \cap \Omega)} p_{\omega} \circ \phi^{-1} \right)$$
where $\Laplace$ denotes the generalised Laplace operator. Now let $f \in C^\infty_c(\mathbb{C})$ then $f \circ \phi \in C^\infty_c(\overline{\Omega})$ and
\begin{align}
 2 \pi \int_\mathbb{C} f \; \d \left( \frac{1}{\omega(K \cap \Omega)} \phi_{*} \left(\omega \, \vline_{ \; \overline{\Omega}} \right) \right) = & \frac{2\pi}{\omega(K \cap \Omega)}  \int_{\overline{\Omega}}\left(f \circ \phi\right) \; \d \omega \\
 = & \frac{1}{\omega(K \cap \Omega)} \int_{\overline{\Omega}} \left(f \circ \phi\right) \; \Laplace p_{\omega} \\
= & \int_\mathbb{C} f \; \Laplace\left( \frac{1}{\omega(K \cap \Omega)} p_{\omega} \circ \phi^{-1} \right) \\ =&
\int_\mathbb{C} f \; \Laplace p_{\widehat{\omega}} = 2 \pi \int_\mathbb{C} f \; \d \widehat{\omega} 
\end{align}
So $\widehat{\omega} = \frac{1}{\omega(K \cap \Omega)} \phi_{*} \left(\omega \, \vline_{ \; \overline{\Omega}} \right)$. 
\end{proof}

Theorem $\ref{theorem_2}$ is an immediate consequence of the above theorem 
together with the fact that each $(\Omega_j, \phi_j, \whK_j)$, $j=1,2$ is a conformal renormalization of $K$.

\REFTHM{general_renorm} 
Let $K$ be a non-thin compact set and let $\Om$ be a domain in $\C$ with boundary a finite collection of extended ray-pairs 
$(\Ray_{t_1}^e, \Ray_{t_2}^e)$, \ldots, $(\Ray_{t_{2n-1}}^e, \Ray_{t_{2n}}^e)$ where 
$t_1<t_2<t_3<\ldots < t_{2n-1} < t_{2n} < 1+t_1$. 
Then there exists a univalent map {\mapfromto \phi \Om \C} such that $(\Om, \phi, \phi(K\cap\Om))$ 
is a normalized conformal renormalization of $K$.
\ENDTHM

Before we proceed with the proof let us prove the following generalization of \lemref{conform_equivalent}. 
Let $K$ be a non-thin, non-connected compact set with $D$ the connected component of $\Chat\Sm K$ containing $\infty$. 
Let $z_0$ be a critical point for the Green's function $g$ for $D$ and suppose $(\Ray_{t_1}^e, \Ray_{t_2}^e)$, 
{\mapfromto {\Ray_{t_j}^e} {[l_0, \infty[} D}, $j= 1,2$ is an extended ray-pair with $t_1 < t_2 < 1+t_1$ and 
$z_0 = R_{t_1}^e(l_0) = R_{t_2}^e(l_0)$. Write $\Ga:= \Ray_{t_1}^e\cup\Ray_{t_2}^e$ and
let $\Om_1, \Om_2$ be the connected components of $\C\Sm\Ga$ containing the rays $\Ray_t$ 
with $t_1<t<t_2$ respectively $t_2<t<1+t_1$. 

Consider $\om\subseteq \C$ an open neighbourhood of  $\Ga$.
Define $\om_j:=\om\cap\Om_j$ and for each $j=1,2$. 
Let $\sim_j$ denote the smallest equivalence relation on 
$\oom_j= \om_j\cup\Ga$ for which $\Ray_{t_1}^ e(l) \sim_j \Ray_{t_2}^ e(l)$ for every $l\geq l_0$. 
Let $X_j:=\oom_j/\sim_j$ denote the quotient space 
and {\mapfromto {\pi_j} {\oom_j} {X_j}} the canonical projection. 
Set $R_j:= \pi_j(\Ga)$ and equip $X_j$ with the quotient topology. Then
\REFLEM{gluing_ext_Ray_pairs}
The topological space $X_j$ has a unique complex structure for which the restriction 
{\mapfromto {\pi_j} {\om_j} {X_j\Sm R_j}} is biholomorphic. 
\ENDLEM
\PROOF
Possibly reducing $\om$ we can suppose $\om$ is a simply connected subset on which $g(z) > l_0/2$. 
Let {\mapfromto h \om \R} denote the harmonic conjugate of $g$ on $\om$ normalized by $h(z_0) = 0$. 
Then $h\equiv 0$ on $\Ga$. Define $f=g+ih$ on $\om$, then possibly reducing $\om$ further 
we can suppose all critical values of $f$ are contained in the line $[l_0, \infty[$, since the critical points of $f$ 
are the critical points of $g$ in $\om$ and $g$ has only finitely many critical points with critical value at least $l_0/2$. 
Let $l_0 < l_1 < \ldots < l_n$ denote the critical values of $f$ and $l_{n+1}:=\infty$. 
Write $I_k := ]l_k, l_{k+1}[$ for $k=1,\ldots , n$. 
And for $\eps>0$ write $I_\eps := ]-\eps, \eps [$ and $I_\eps(l) := l+I_\eps$for $l\in\R$. 

By compactness of $\Ray_{t_j}^e([l_0, l_n])$ there exists $\eps>0$ such that
\ENUM
\item
The connected component $\De^0$ of $f^{-1}(I_\eps(l_0)\times I_\eps)$ containing $z_0$ 
maps properly onto $I_\eps(l_0)\times I_\eps$ with degree $\deg(f, z_0)$ and unique critical value $l_0$.
\item
For $j=1,2$ and $k = 1, \ldots , n$ the connected component $\De^k_j$ of $f^{-1}(I_\eps(l_k)\times I_\eps)$ 
containing $\Ray_{t_j}(l_k)$ maps properly onto $I_\eps(l_k)\times I_\eps$ 
with degree $\deg(f, \Ray_{t_j}(l_k))$ and if the degree is greater than $1$ unique critical value $l_k$.
\item
For $j=1,2$ and $k = 0, \ldots , n-1$ the connected component $\Xi^k_j$ 
of $f^{-1}(I_k\times I_\eps)$ containing $\Ray_{t_j}(I_k)$  
maps isomorphically onto $I_k\times I_\eps$. 
\item
For $j=1,2$ the connected components $\Xi^n_j$ of  $f^{-1}(I_n\times\infty)$ 
containing $\Ray_{t_j}(I_n)$ maps isomorphically onto $f(\Xi^n_j)\supset I_n$.
\ENDENUM
We define $\whDe_j^0:= \pi_j(\De^0)$, $\whDe^k_j := \pi_j((\De_1^k\cup\De_2^k)\cap\oom_j)$ and 
$\whXi_j^k:=\pi_j(\Xi_1^k\cup\Xi_2^k)\cap\oom_j)$.
For the existence of the complex structures fix $j\in\{1,2\}$. 
We need to provide for each $w\in R_j$ a local chart around 
$w$ compatible with {\mapfromto {\pi_j^{-1}} {X_j\Sm R_j} {\om_j}}.
For  $w = R_j(l)$ with $l\in I_k$ for some $k$ the open set $\whXi_j^k$ is a neighbourhood of $w$ 
and $f\circ\pi_j^{-1}$ is a local chart compatible with $\pi^{-1}$. 
For $w = R_j(l_k)$ with $0<k\leq n$, the set $\whDe^k_j$ is an open neighbourhood of $w$. 
Moreover for $j=1$ the image by $f$ of $\De_1^k$ covers the set $I_\eps(l_k)\times[0,\eps[$ $1+m_{11}^k$ times 
and the set $I_\eps(l_k)\times]-\eps, 0[$ $m_{11}^k$ times, where $m_{11}^k\geq 0$ and 
the image by $f$ of $\De_2^k$ covers the set $I_\eps(l_k)\times]-\eps, 0]$ $1+m_{12}^k$ times 
and the set $I_\eps(l_k)\times]0,\eps[$ $m_{12}^k$ times, where $m_{12}^k\geq 0$.
Hence {\mapfromto {f(\pi_1^{-1}}{\whDe_1^k}{I_\eps(l_k)\times I_\eps}} is a topological branched covering of
degree $d_1^k:= 1+m_{11}^k+m_{12}^k$ with unique critical value $l_k$ and thus $(f(\pi_1^{-1})-l_k)^{1/d_1^k}$ 
is a local chart holomorphically compatible with $\pi^{-1}$. And similarly for $j=2$. 
(Note that since $l_k$ is a critical value of $f$ we have $m_{11}^k+m_{12}^k+m_{21}^k+m_{22}^k\geq 1$.)
Finally for $w=\pi_j(z_0)= R_j(l_0)$ the set $\whDe_j^0$ is an open neighbourhood of $w$ and 
the image of $\De^0\cap\oom_j$ by $f$ covers $]l_0,l_0+\eps[$ $1+d_j^0$ times and 
$I_\eps(l_0)\times I_\eps\Sm [l_0,l_0+\eps[$ $d_j^0$ times so that 
{\mapfromto {f(\pi^{-1})} {\whDe_j^0} {I_\eps(l_0)\times I_\eps}} is a topological degree $d_j^0$ branched covering 
with unique critical value $l_0$ and thus $(f(\pi_1^{-1})-l_0)^{1/d_j^0}$ 
is a local chart holomorphically compatible with $\pi^{-1}$. 
Clearly any pair of these charts are also mutually holomorphically compatible. 
This proves existence. 

For the uniqueness suppose {\mapfromto \chi \Si \C} is a chart around $w\in R_j$ 
holomorphically compatible with $\pi_j^{-1}$. Let $\eta$ be one of the above charts defined around $w$, 
the $\chi\circ\eta$ is holomorphic off the real line by construction. 
However as the real line is holomorphically removeable $\chi\circ\eta$ is everywhere holomorphical and 
so belongs to the same complex structure as defined above.
\ENDPROOF

\begin{proof}[Proof of theorem \ref{general_renorm}]

Define $\sim$ to be the smallest equivalence relation on $\overline{\Omega}$ such that for any $j\in\{1, \ldots, n\}$ 
and any $x \in \mathcal{R}_{t_{2j-1}}$, $y \in \mathcal{R}_{t_{2j}}$ with $g(x) = g(y)$ : $x\sim y$. 
Then $X:=\overline{\Omega}/{\sim}$ equipped with the quotient topology is compact, 
connected, simply connected and Hausdorff.
Let {\mapfromto \pi {\ov{\Om}} X} denote the canonical projection.

\noindent\textbf{Claim 1:}
The space $X$ has a unique complex structure for which the restriction
 {\mapfromto \pi {\Om} {X\Sm\pi(\bd{\Om})}} is biholomorphic.
 
\noindent\textbf{Proof of Claim 1:} Indeed by \lemref{general_renorm} above we need only check that $\pi(\infty)$ is a puncture. 
 To this end we define a conformal isomorphism of a neighbourhood of $\pi(\infty)$ onto a neighbourhood of $\infty$ 
 compatible with $\pi$ and hence with the other local charts. 
 Let {\mapfromto \psi U V} be as in \lemref{existence_and_uniqueness_of_psi}. 
 Then $\Om\cap U$ has $n$ connected components mapping univalently by $\psi$ to the $n$ sectors in $V$ 
 consisting of points with arguments in the intervals $]t_{2j}, t_{(2j+1)\mod 2n}[$. 
 Let $U' :=\pi(\ov{\Om}\cap U)$
 Denote by $\Si_j$ the component of $\Om\cap U$ intersecting $\Ray_t$ with $t\in ]t_{2j}, t_{(2j+1)\mod 2n}[$ 
 and write $\De_j := (t_{(2j+1)\!\!\!\mod\! 2n}-t_{2j})$.
 Let $\De = \sum_{j=1}^n{\De_j} = \om(K\cap\Om)$, let $\tau_1 = 0$ and let $\tau_j = \sum_{k=1}^{j-1}\De_k$ for $j=2, \ldots , n$. 
Define a homeomorphism {\mapfromto \varphi {\pi(U')} {V':=\{ z | \log|z| > L/\De\}}} by
$$
\varphi(\pi(z)) = \e^{i2\pi\tau_j/\Delta}\left(\frac{\psi(z)}{\e^{i2\pi t_{2j}}}\right)^{1/\De}\qquad z\in\ov{\Si}_j,
$$
where the root is the principal root, i.e.
$$
\left(\frac{\psi(z)}{\e^{i2\pi t_{2j}}}\right)^{1/\De} := \exp\left(\frac{g(z)+i(h(z)-2\pi t_{2j})}{\De}\right).
$$
Then $\varphi$ is holomorphic for the complex structure already defined on $X\Sm\pi(\infty)$ 
and so defines a chart around $\pi(\infty)$ sending $\infty$ to $\infty$. 
Since there are charts around each scar $\pi(\Ray^e_{t_{2j-1}})=\pi(\Ray^e_{2j})$ mapping the scar to a line segment, 
the scars are holomorphically removeable. 
And so any homeomorphism {\mapfromto \eta W {\eta(W)\subset\C}} holomorphically compatible with $\pi$ 
and defined near a point $w_0$ on some scar or the point $\pi(\infty)$ is holomorphically compatible with the charts defined above. 
Thus the complex structure is uniquely defined by $\pi$.

Let {\mapfromto \Phi X \Chat} denote the isomorphism normalized by 
\begin{align}
\Phi(\pi(z)) = \varphi(\pi(z))+ \oo(1)
\qquad\textrm{as}\qquad z \to \infty, \label{normalise_gen_conf_rocs}
\end{align}
so that in particular $\Phi(\pi(\infty))=\infty$. 
Then {\mapfromto \phi \Om \C} give by $\phi:=\Phi\circ\pi$ is univalent. Let $\whK := \phi(\Om\cap K)$ and let 
$\whR_j:= \Phi(\pi((\Ray^e_{t_{2j-1}}))=\Phi(\pi(\Ray^e_{2j}))$.
Then

\noindent\textbf{Claim 2:} 
The triple $(\Om, \phi, \whK)$ is a conformal renormalization of $K$.

\noindent\textbf{Proof of Claim 2:} The proof is verbatim the same as that of \thmref{greensfunctions_related}
\end{proof}

\EXA
Let $g$ be the Green's function defined by a non-thin compact set $K$ and suppose $g$ has a unique critical point
$z_0$ of maximal potential $L=g(z_0)$. 
Let $\Ray_{t_1}$ and $\Ray_{t_2}$ with $t_1 < t_2 < 1 + t_1$ be two smooth rays landing at $z_0$ 
and let {\mapfromto {\Ray_{t_j}^e} {[L,\infty[} \C} be the unique extensions to $[L,\infty[$. 
Then there exist conformal renormalizations $(\Omega_j, \phi_j, \whK_j)$ where $\Om_j$, $j=1, 2$
are the connected components of $\C\Sm(\Ray_{t_1}^e\cup\Ray_{t_2}^e)$. 
If the order of the critical point $z_0$ is greater than one then at least for one of the two, say for  $j=1$ the image
$\phi_j(\Ray_{t_2}^e)$ is an extended ray with endpoint a critical point of lower order.
\ENDEXA

Next we will consider the case of an arbitrary number of critical points of maximal potential and a region $\Omega$ bounded by some finite number of co-landing pairs of rays. An example of such a case is sketched in figure \ref{fig:two_critical_points_of_max_potential}. 
\EXA
Let $g$ be a Green's function defined by a non-thin compact set $K$ and suppose $z_1, ..., z_m$, $m\geq 1$ are distinct critical points of maximal potential $L = g(z_1) = \ldots = g(z_m)$ . 
For $j \in \{1, ..., n\}$ let $\Ray_{t_{2j-1}}$ and $\Ray_{t_{2j}}$ be a pair of smooth rays co-landing at $z_{k(j)}$ with $t_1< t_2<\ldots< t_{2n-1} < t_{2n} < 1 + t_1$. 
Let  {\mapfromto {\Ray_{t_j}^e} {[L,\infty[} \C} be the unique such extensions for each $j$. 
Then for each connected component $\Om$ of 
$\C\Sm \left(\bigcup_{j=1}^{2n} \Ray_j^e \right)$
there exists a conformal isomorphism {\mapfromto \phi\Om \C} such that $(\Om, \phi, \whK)$, where $\whK:= \phi(K)$ 
is a conformal renormalization of $K$.
\ENDEXA

\begin{figure}[h]
 \centering
  \includegraphics{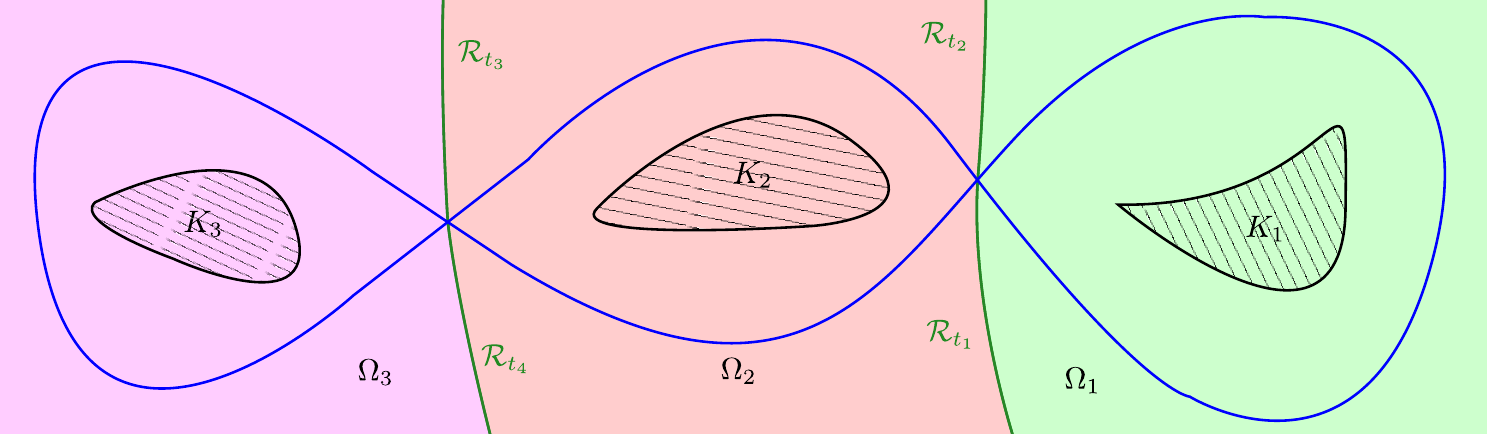} 
  \caption{An illustration of a Green's function with two simple critical points of maximal potential. In black the boundary of the domain $D$, in blue the equipotential curves $L_{g(z_1)}=L_{g(z_2)}$, in dark green the orthogonal rays, $\mathcal{R}_{t_1},\mathcal{R}_{t_2},\mathcal{R}_{t_3},\mathcal{R}_{t_4}$ landing at $z_2$ and $z_1$ respectively, in light green, light red and light purple the domains $\Omega_1, \Omega_2$ and $\Omega_3$ respectively.} 
\label{fig:two_critical_points_of_max_potential}
\end{figure}

\subsection{Iterated renormalization}

Conformal renormalization behaves well under iteration. 
\REFPROP{iterated_renorm}
Suppose $( \Omega, \phi, \whK)$ is a conformal renormalization of a non-thin compact set $K$. 
And suppose further that $(\widehat{\Omega}, \whphi, \widehat{\whK})$ is a conformal renormalization of $\whK$. 
Define $\widehat{\whOm}:=\phi^{-1}(\whOm)$ and {\mapfromto {\widehat{\whphi}:=\whphi\circ\phi} {\widehat{\whOm}} \C}, 
then $(\widehat{\whOm}, \widehat{\whphi}, \widehat{\whK})$ is a conformal renormalization of $K$. 
\ENDPROP

\PROOF
Clearly points (1)-(3) of \defref{def:conf_renormalization} are fullfilled. 
For (4) let $g$ denote the Green's function defined by $K$, $\whg$ the Green's function defined by $\whK$ and 
$\widehat{\whg}$ the Green's function defined by $\widehat{\whK}$. 
And let $\om$ denote the equilibrium measure for $K$, $\whom$ the equilibrium measure for $\whK$ and 
$\widehat{\whom}$ the equilibrium measure for $\widehat{\whK}$.
Then
\ALIGN
\widehat{\widehat{g}}(z) 
&= \frac{1}{\widehat{\omega}\left(\whK \cap \widehat{\Omega}\right)}\widehat{g}\left(\widehat{\phi}^{-1}(z)\right)
= \frac{1}{\widehat{\omega}\left(\whK \cap \widehat{\Omega}\right) \omega\left(K \cap \Omega\right)} g\left(\phi^{-1}\left(\widehat{\phi}^{-1}(z)\right)\right)\\
&= \frac{1}{\omega\left(K \cap \widehat{\whOm}\right)}g\left(\widehat{\whphi}^{-1}(z)\right).
\end{align*}
Thus $(\widehat{\whOm}, \widehat{\whphi}, \widehat{\whK})$ is a conformal renormalization of $K$.
\ENDPROOF

\subsection{Direct deeper renormalizations}

Towards a proof of \thmref{thm_1}. 
Let $g$ be the Green's function defined by a non-thin compact set $K$. 
And let $W$ be a Jordan domain with smooth boundary, along which $\napla g$ points out of $W$, 
i.e. let {\mapfromto \de {[a,b]} {\bd{W}}} be a regular ($\de'$ never vanishing) parametrization, 
then $|\Arg(i\napla g(\de(t))/\de'(t))|<\pi/2$ for every $t\in[a,b]$. 
\REFPROP{W_uniqueness}
Suppose {\mapfromto \phi W \C} is a univalent map such that the Green's function $\whg$ 
defined by $\whK:= \phi(W\cap K)$ satisfies
$$
\widehat{g}(z)=\frac{1}{\omega(K \cap \Omega)}g\left(\phi^{-1}(z)\right).
$$
Then $\phi$ is unique up to post-composition by an affine map and equivalently $\whK$ is unique up to affine transformation.
\ENDPROP
\PROOF
Let {\mapfromto {\phi_j} W \C}, $j=1,2$ be two univalent maps with $g(z) = \whg_j(\phi_j(z))$, 
where $\whg_j$ is the Green's function defined by $\whK_j := \phi_j(K)$. 
Let $\om$ denote the equilibrium measure on $K$ and $\whom_j$ the equilibrium measure on $\whK_j$. 
Then $\whom_j = (\phi_j)_*(\om|_{(K\cap W)})/\om(K\cap W)$ by a proof similar to that of  \thmref{push_forward_measure}.
Thus $\whg_j$ has no critical point outside $W_j$ and in fact not in a neighbourhood of $\bd{W}_j$ by 
the hypothesis $\napla g$ points out of $W$. 
Hence the univalent maps {\mapfromto {\psi_j = \exp(\whg_j+i\whh_j}{U_j}{V_j}} of \lemref{existence_and_uniqueness_of_psi} 
univalently extend to a neighbourhood of $\Chat\Sm W_j$ for each $j$. 
Moreover $h-(\om(K\cap W)\cdot\whh_j\circ\phi_j)$ are constant functions where defined. 
Thus possibly adding a constant to say $\whh_2$ and adjusting domains we can suppose 
$\om(K\cap W)\cdot\whh_1\circ\phi_1=\om(K\cap W)\cdot\whh_2\circ\phi_2$ on $W\Sm W'$, 
where $K\cap W \subset W'$ and $\ov{W}'\subset W$. 
Let $\whW'_j:=\phi_j(W')$. Then 
{\mapfromto {\phi_2\circ\phi_1^{-1} = \psi_2\circ\psi_1^{-1}}{\whW_1\Sm\whW_1'}{\whW_1\Sm\whW_1'}} 
is an isomorphism. Hence we can define an isomorphism {\mapfromto \eta \Chat \Chat} fixing $\infty$ by 
$$
\eta(z) = 
\begin{cases}
\phi_2\circ\phi_1^{-1}(z),\qquad & z\in\whW_1,\\
\psi_2\circ\psi_1^{-1}(z),\qquad & z\not\in\whW_1.
\end{cases}
$$
Since $\eta(\infty) = \infty$, the map $\eta$ is the desired affine map such that $\phi_2 = \eta\circ\phi_1$ and 
$\whK_2=\eta(\whK_1)$.
\ENDPROOF

\subsection{Proof of \thmref{thm_1}.}
Let $g$ be the Green's function defined by a non-thin compact set $K$. Let $l>0$ be arbitrary 
and let $W$ be a bounded connected component of $\C\Sm g^{-1}(l)$. 
Let $I\subset \T$ denote the set of arguments $t$ such that $\Ray_t\cap W\not=\emptyset$. 
Then either $I=\T$ and $W$ contains $K$ or $I$ is a finite union of open intervals. 
Let $t_1< t_2 < \ldots < t_{2n} < 1+t_1$ be the endpoints of these intervals such that 
$$
I = \bigcup_{j=1}^n ]t_{2j}, t_{2j+1\!\!\!\mod\! 2n}[
$$
Then there are $n$ distinct critical points $z_1, \ldots z_n$ and pairs of extend rays 
$(\Ray_{t_{2j-1}}^e, \Ray_{t_{2j}}^e)$ co-landing on $z_j$ and bounding a domain $\Om$ 
such that $W\subset\Om$ and $W\cap K = \Om\cap K$.
Let $(\Om, \phi, \whK)$ denote the corresponding conformal renormalization of $K$, 
whose existence is assured by \thmref{general_renorm}. 
Then by \thmref{push_forward_measure} the restriction of $\phi$ to $W$ fulfills the requirements of \thmref{thm_1} for $W$. 
Moreover by \propref{W_uniqueness} the map $\phi$ is unique up to post composition by an affine map. 
Since $W$ was an arbitrary connected component of $\C\Sm g^{-1}(l)$. 
We have proved \thmref{thm_1}. 

By \thmref{thm_1} each bounded connected component of $\C\Sm g^{-1}(l)$ induces a distinct conformal renormalization. The next theorem gives an easy way to determine the number of bounded connected components of $\C \setminus g^{-1}(l)$.
\begin{theorem}
Let $l>0$, if $g$ has $m$ critical points of order $n_1, n_2, ..., n_m$ contained in $\{z | g(z) \geq l\}$, then $\C \setminus g^{-1}(l)$ consists of $1 + \sum_{j=1}^m n_j $ bounded connected components.
\end{theorem}
\begin{proof}[Proof by induction.]

The base case, which is where $m=1$ follows from the fact that the critical equipotential level 
is topologically a bouquet of $n_1 + 1$ circles. 
Now assume the induction hypothesis for all $k \leq m$. 
Suppose there is $m+1$ critical points $z_1, z_2, \ldots, z_{m+1}$ of order $n_1,  \ldots, n_m, n_{m+1}$ 
and assume without loss of generality that the $z_{m+1}$ is of maximal potential. 
Then there is a total of $n_{m+1}+1$ rays $\mathcal{R}_{t_1}, \mathcal{R}_{t_2}, ... , \mathcal{R}_{t_{n_{m+1}+1}}$ 
landing on the critical point $z_{m+1}$. 
These rays divides $\Chat$ into $\Omega_1, \Omega_2, \ldots, \Omega_{n_{m+1}+1}$ 
each inducing conformal renormalizations $(\Omega_j, \phi_j, \whK_j )$ of $K$ 
by theorem \ref{general_renorm}.  

Since the rays $\mathcal{R}_{t_1}, \mathcal{R}_{t_2}, ... , \mathcal{R}_{t_{n_{m+1}+1}}$ all land at $z_{m+1}$ then $z_{m+1} \not \in \bigcup_{j=1}^{n_{m+1}+1} \Omega_j$. On the other hand since $z_{m+1}$ is of maximal potential then for $ 1 \leq k \leq m$ we have that $z_k \not\in \bigcup_{j=1}^{n_{m+1}+1} \overline{\mathcal{R}}_{t_j}$ so $z_k \in \bigcup_{j=1}^{n_{m+1}+1}  \left(\Omega_j \cap \{z | g(z) \geq l\} \right)$. Since $\Omega_1, \ldots, \Omega_{m+1}$ are disjoint then each $z_k$ is contained in exactly one of the domains. If $z_k \in \Omega_j$ then since $\phi_j$ is univalent on $\Omega_j$ we have that $\phi_j(z_k)$ is a critical point of order $n_k$. 

Therefore if we let $\widehat{g}_j$ denote the Green's function for $\whK_j$, let $\widehat{m}_j$ denote the number of critical points of $\widehat{g}_j$ and let $\widehat{n}_1, \ldots, \widehat{n}_{\widehat{m}_j}$ denote their order, 
then $\sum_{j=1}^{n_{m+1} + 1} \sum_{k=1}^{\widehat{m}_j} \widehat{n}_k = \sum_{j=1}^m n_k$ 
and $\sum_{j=1}^{n_{m+1}+1} m_j = m$ so in particular $\widehat{m}_j \leq m$.  
Hence by the induction hypothesis $\C \setminus \widehat{g}^{-1}(\frac{l}{\omega_K (\Omega_j \cap K)})$ 
consists of $1 + \sum_{k=1}^{m_j} \widehat{n}_k $ bounded connected components. 
Therefore $g$ consists of 
$\sum_{j=1}^{n_{m+1}+1} \left( 1 + \sum_{k=1}^{m_j} n_k \right) = 1 + n_{m+1} + \sum_k^m n_k = 1 + \sum_{k=1}^{m+1} n_k$ bounded connected components which completes the induction.
\end{proof}

\section{Inverse renormalization}

In this section, we will construct an inverse procedure to the conformal renormalization scheme in the special case where the Green's function $g$ for a non-thin compact set $K$ has a unique simple critical point of maximal potential. More precisely, given two non-thin compact sets $K_1, K_2 \subseteq \C$, a choice of weight distribution $\Delta_1, \Delta_2$ with $\Delta_1 +\Delta_2 = 1$ and for $L$ sufficently large (Depending on $K_1$ and $K_2$), we will construct a compact set $K \subseteq \C$, such that $K_1$ and $K_2$ are conformal renormalizations of $K$. 

Our construction makes use of the notion of prime ends from Carathéodory Theory. For an introduction to prime ends see \cite{milnor}. For completeness we include the following basic facts.

Let $U$ be an open and simply connected subset of $\Chat$ conformally equivalent to the unit disc $\mathbb{D}$. A \emph{cross cut} in $U$ is defined as an arc $A: [0,1] \to \overline{U}$ which intersects the boundary $\partial U$ only at the distinct end-points. Any cross cut $A$ cuts $U$ into two connected components. Choose some base point $b_0 \in U$, then for any cross cut which is disjoint from $b_0$ we define the neighborhood $N(A)$ to be the component of $U \setminus A$ which does not contain $b_0$. 

We define a \emph{chain} to be an infinite sequence $A_1, A_2, A_3, ...$ of disjoint crosscuts with the property that the corresponding neighborhoods are nested, that is
$$
N(A_1) \supset N(A_2) \supset N(A_3) \supset ...
$$
and such that the spherical diameter of $A_n$ converges to $0$ as $n \to \infty$.

We say that two chains $\{ A_j \}_{j > 0}$ and $\{ A_j' \}_{j > 0}$ are equivalent if every $N(A_j)$ contains some $N(A_j')$ and every $N(A_j')$ contains some $N(A_j)$. An equivalence class of chains is called a \emph{prime end} $\mathcal{E}$ in $U$. We call the set of all prime ends in $U$ for the \emph{Caratheodory boundary} of $U$, and therefore a prime end is also called a point in the Carathéodory boundary of $U$.

We can compactify $U$ by the use of prime ends in $U$. The \emph{Carathéodory completion} $\hat{U}$ of $U$ is defined as the disjoint union of the set $U$ and the Caratheodory boundary of $U$, equipped with the following topology.

For any cross cut $A \subset \overline{U} \setminus {b_0}$ the neighborhood $\mathcal{N}_A$ is defined to be the neighborhood $N(A) \subset U$ together with the set consisting of all prime ends $\mathcal{E}$ for which some representative chain $\{A_j\}_{j>0}$ satisfies $A_1 = A$. These neighborhoods $\mathcal{N}_A$, together with the open subsets of $U$ form a basis for the Carathéodory topology.

By Caratheodorys theorem any Riemann map then extends to a homeomorphism from $\overline{\mathbb{D}}$ to $\hat{U}$ so the boundary of the Carathéodory completion is homeomorphic to $\partial \mathbb{D}$. 

Returning to the construction of the inverse renormalizaion: Take the disjoint union of two copies of the Riemann sphere, $\Chat_1 \sqcup \Chat_2$, and let  $K_1 \subset \Chat_1 \setminus \{ \infty_1 \}$ and $K_2 \subset \Chat_2 \setminus \{ \infty_2 \}$ be two full non-thin compact sets, and let $D_1 = \Chat_1 \setminus K_1$ and $D_2 = \Chat_2 \setminus K_2$ be their dual domains. Now for for $j \in \{1,2\}$, let $g_j$ denote the Green's function for $D_j$ and choose points $z_j \in D_j \setminus \{\infty\}$ such that $g_j(z_j) > \max \{v | v\in\sing(g_j)\}$.

Define $\Delta_1= \frac{g_2(z_2)}{g_1(z_1) + g_2(z_2)}$ and $\Delta_2= \frac{g_1(z_1)}{g_1(z_1) + g_2(z_2)}$ so that $\Delta_1 + \Delta_2 = 1$ and  $L:=\Delta_1 g_1(z_1) = \Delta_2 g_2(z_2)$.

Let $\mathcal{R}^j_t \subset D_j$ be the unique smooth external ray containing $z_j$, 
let $\mathcal{R}^j = \{z \in \mathcal{R}^j_t | g_j(z) > g_j(z_j) \} \cup \{ \infty \}$ 
and define $Y_j := \Chat_1 \setminus \mathcal{R}^j$. 
Now let $\whY_j$ denote the Carathodory completion of $Y_j$ and equip it with the Carathodory topology. We will denote the homeomorphic extension of the Riemann map by $\phi_j: \overline{\mathbb{D}} \to \whY_j$.

Since $\partial Y_j$ is an analytic arc from $z_j$ to $\infty$ then $z_j, \infty \in \partial Y_j$ corresponds to one point in the Carathéodory boundary, and every other boundary point, $z \in \partial Y_j \setminus \{z_0, \infty\}$, corresponds to two points in the Carathéodory boundary.

We define $\mathcal{R}^j_+$ to be the image of the clock-wise arc of $\mathbb{S} = \partial\mathbb{D}$ going from $\phi_j(z_0)$ to $\phi_j(\infty)$ (endpoints included) and $\mathcal{R}^j_-$ to be the image of the counter clock-wise arc of $\mathbb{S}$ going from $\phi_j(z_0)$ to $\phi_j(\infty)$ (endpoints included). We extend $g_j$ continously to $\whY_j$.

We can now glue $\partial \whY_1$ to $\partial \whY_2$ by matching scaled potentials. Let $\sim$ denote the smallest equivalence relation on $Y=\whY_1 \sqcup \whY_2$ where $x \sim y$ if 
\begin{itemize}
\item $x \in \mathcal{R}^1_+ , y \in \mathcal{R}^2_+$ and $\Delta_1 g_1(x)= \Delta_2 g_2(y)$, or
\item if $x \in \mathcal{R}^1_- , y \in \mathcal{R}^2_-$ and $\Delta_1 g_1(x)= \Delta_2 g_2(y)$.
\end{itemize}
Equip $Y$ with the disjoint union topology and equip $X:=Y/{\sim}$ with the quotient topology, so that the canonical projection $\pi:Y \to X$ is a quotient map. It is straight forward to show that $X$ is compact, simply connected and Hausdorff.

We will now equip $X$ with a complex structure making it conformally equivalent to the Riemann sphere.

\begin{lemma} \label{lemma_inverse_glue}
	There exists a unique complex structure on $X$ such that the restricted projection $\pi\vline_{\,Y_j}: Y_j \to X$ is an injective holomorphic mapping for $j \in \{1,2\}$.
\end{lemma}
\begin{proof}
We can use the inverse projection restricted to $Y_j $ for $j=1,2$ as the first two charts. We can cover $\pi(\infty_1)=\pi(\infty_2)$ aswell as $\pi(\mathcal{R}^1_+ \cup \mathcal{R}^1_-) \setminus \pi(z_1) = \pi(\mathcal{R}^2_+ \cup \mathcal{R}^2_-) \setminus \pi(z_2)$ with $f: A \to B$
defined by $A=\{z | (z \in \whY_1 \land g_1(z) > g_1(z_1)) \lor (z \in \whY_2 \land g_2(z) > g_2(z_2)) \}$, $B=\{ z \in \mathbb{C} | |z| < \frac{1}{L} \}$ and 
$$
f(z) = \begin{cases} (\psi_1({\pi\vline_{\,\whY_1}}^{-1}(z)))^{-\Delta_1} \quad & \text{for } z \in \pi(\whY_1) \\
(\psi_2({\pi\vline_{\,\whY_2}}^{-1}(z)))^{-\Delta_2} \quad & \text{for } z \in \pi(\whY_2)
\end{cases}
$$
choosing a suitable branch such that $f$ is a well-defined homeomorphism.
It now remains to cover $\pi(z_1)=\pi(z_2)$.

Assume without loss of generalisation that $z_1 = 0$ and take some sufficently small simply connected open neighbourhood $O_1 \subset D_1$ of $z_0$ and let $h_1$ be the harmonic conjugate to $g_1$, normalised such that $h_1(z)=0$ on $\mathcal{R}^1$. Let $f_1(z)= g_1(z_0)-g_1(z) + ih_1(z)$ then since $f_1(O_1)$ is open it contains some disc $B_{\varepsilon_1}$ of radius $\varepsilon_1$ centered around $f_1(z_1)=0$. Choosing a suitable branch of the square root function we can map $B_{\varepsilon_1} \setminus f_1(\mathcal{R}^1)$ onto the half disc $\{z | |z| < \sqrt{\varepsilon_1} \land \mathrm{Re}(z) > 0 \}$. The set $B_{\varepsilon_1} \setminus f_1(\mathcal{R}^1)$ can be regarded as a subset of $\whY_1$ and evidently the mapping $(f_1(z))^{1/2}$ can be extended continuously to $\mathcal{R}^1_+$ and $\mathcal{R}^1_-$ mapping them onto $\{z | \mathrm{Re}(z) = 0 \land 0 \leq \Im (z) < \sqrt{\varepsilon_1} \}$ and $\{z | \mathrm{Re}(z) = 0 \land 0 \geq \Im (z) > - \sqrt{\varepsilon_1} \}$ respectively.

Similarly, assume without loss of generalisation that $z_2=0$ and take a sufficently small  neighbourhood $O_2 \subset D_2$, and define $f_2$ such that $f_2(O_2)$ contains some disc $B_{\varepsilon_2}$ of radius $\varepsilon_2$ centered around $f_2(w_0)=0$. Again choosing the other branch of the square root function we can map $B_{\varepsilon_2} \setminus f_2(\mathcal{R}^2)$ onto the other half disc $\{z | |z| < \sqrt{\varepsilon_2} \land \mathrm{Re}(z) < 0 \}$ including a continously extension of $\mathcal{R}^2_+$ and $\mathcal{R}^2_-$ onto $\{z | \mathrm{Re}(z) = 0 \land 0> \Im (z) \geq - \sqrt{\varepsilon_2} \}$ and $\{z | \mathrm{Re}(z) = 0 \land 0 \leq \Im (z) < \sqrt{\varepsilon_2} \}$ respectively.

Take $\varepsilon=\mathrm{min} \{\varepsilon_1, \varepsilon_2\}$ then the mapping $k: \pi(A_1 \cup A_2) \to B_\varepsilon$ defined by $A_1=f_1^{-1}(B_\varepsilon) \setminus \mathcal{R}^1 \cup \{z | (z \in \mathcal{R}^1_+ \lor z \in \mathcal{R}^1_-) \land |g_1(z) - g_1(z_0)| < \varepsilon \}$, $A_2=f_2^{-1}(B_\varepsilon) \setminus \mathcal{R}^2 \cup \{z | (z \in \mathcal{R}^2_+ \lor z \in \mathcal{R}^2_-) \land |g_2(z) - g_2(w_0)| < \varepsilon   \}$ and 
$$k(z)=\begin{cases} (f_1(\pi^{-1}(z)))^{1/2} \quad & \text{for } \pi(A_1)  \\
(f_2(\pi^{-1}(z)))^{1/2}  \quad & \text{for } \pi(A_2)
\end{cases} $$
identifies exactly those points which $\sim$ deems equivalent, so it is  a well-defined homeomorphism.

Clearly these charts make up a complex structure on $X$.

Let $R = \pi( \mathcal{R}^1_+ \cup \mathcal{R}^1_-)=\pi(\mathcal{R}^2_+ \cup \mathcal{R}^2_-)$ and suppose that $\mathcal{A}$ and $\mathcal{A}'$ are two complex structures on $X$, for which $\pi_j\;\vline_{\;Y_j}$ is an injective holomorphic mapping, then any chart on $X \setminus R$ would clearly have to be compatible.

Suppose that $\phi_1$ and $\phi_2$ are charts in $\mathcal{A}$ and $\mathcal{A}'$ respectively, both covering $w \in R$. Let $O_w$ be an open neighbourhood of $w$ covered by both charts, then their transition mapping $k:\phi_1(O_w) \to \phi_2(O_w)$ defined by $k = \phi_2 \circ \phi_1^{-1} $ is homeomorphism and it is holomorphic outside the analytic arc $\phi_2(R \cap O_w)$, which is holomorphically removeable. Therefore $k$ extends holomorphically to the entire neighbourhood and so the change of charts is holomorphic.
\end{proof}

It follows that for a sufficently small open neighbourhood $\om \subseteq X$ of $\pi(\infty_1) = \pi(\infty_2)$, 
the function  $\eta: \om \to \Chat$ defined by
$$
\eta(z) = \begin{cases} \psi_1(\pi^{-1}(z))^{\Delta_1} \quad & \text{for } z \in \pi(\whY_1) \cap\om \\
\psi_2(\pi^{-1}(z))^{\Delta_2} & \text{for } z \in \pi(\whY_2) \cap\om
\end{cases}
$$ 
is well-defined and holomorphic for a suitable branch choice. For normalization purposes it is convenient to choose the branches such that for $z \in \eta(\pi(\mathcal{R}^1_+))=\eta(\pi(\mathcal{R}^2_+))$ and $w \in \eta(\pi(\mathcal{R}^1_-))=\eta(\pi(\mathcal{R}^2_-))$ then $\Arg z < \Arg w$.
\begin{proof}[Proof of Theorem \ref{inv_ren_thm}.]
Let $t_1\in\T$ and $C,L, \De_1 > 0$ with $\De_1 < 1$ and non-thin centered compact sets $K_1, K_2$ of capacity $1$ with corresponding Green's functions $g_1$ and $g_2$ satisfying 
$L > \sing(g_1) \De_1$ and $L > \sing(g_2) (1-\De_1)$ be given.

Since $\frac{L}{\De_1} > \sing(g_1)$ and $\frac{L}{1-\De_2} > \sing(g_2)$ there exists $z_1 \in \mathcal{R}_0(K_1)$ and $z_2 \in \mathcal{R}_0(K_2)$ such that $g_1(z_1) = \frac{L}{\Delta_1}$ and  $g_2(z_2) = \frac{L}{(1-\Delta_1)}$.

Using $z_1 \in D_1$ and $z_2 \in D_2$ we shall construct $X$ as described above. It then follows from Lemma \ref{lemma_inverse_glue} that $X$ is a compact and simply connected Riemann surface, so by the Riemann uniformization theorem it is isomorphic to the Riemann sphere. Let $\Phi:X \to \Chat$ be the unique isomorphism satisfying 
\begin{align}
\Phi(z) = C {e^{2 \pi i t_1}}\eta(z) + O\left(\frac{1}{\eta(z)}\right) \quad\quad \text{for } z \to \infty \label{norm_inverse}
\end{align}
and let $\phi = \Phi \circ \pi$.

Then $K := \phi(K_1 \cup K_2)$ is a full compact, non-thin and centered subset of the complex plane with $c(K)=C$. Let $D$ be the dual domain for $K$, then since two external rays lands on $z_0:=\phi(z_1)=\phi(z_2) \in D$ then $z_0$ will be a critical point for $g_D$ with potential $g_D(z_0) = \Delta_1 g_1(z_1) = \Delta_2 g_2 (z_2)$. 

That the external rays landing at the critical point  $z_0$ has external arguments $t_1$ and $t_2 = t_1 + \Delta_1$ follows from the normalization in equation (\ref{norm_inverse}).

Let $(\Omega_1, \widehat{\phi_1}, \whK_1)$ and $(\Omega_2, \widehat{\phi_2}, \whK_2)$ be conformal renormalizations of $K$ constructed as in \textsection \ref{sec:conf_ren}.

Recall that $Y_j := \Chat \setminus \left(\{ z | z \in \mathcal{R}^j_0 | g_j(z) > g_j(z_j) \} \cup \{ \infty \} \right)$ for $j=1,2$. A moments consideration reveals that $\widehat{\phi}_j \circ \phi|_{Y_j}$ can be extended to a biholomorphic map from $\Chat$ to itself which sends $\infty$ to $\infty$, hence $\widehat{\phi}_j(\phi|_{Y_j}(z)) = az + b$. Since $c(\whK_j) = c(K_j)$ then 
$|a| = 1$, and since both $\whK_j$ and $K_j$ are centered then $b = 0$. 
Moreover $\mathcal{R}_0(K_j)$ is send to $\mathcal{R}_0(\whK_j)$ hence $\widehat{\phi}_j \circ \phi|_{Y_j}$ must be the identity.
\end{proof}

\end{document}